\documentclass[oneside,english]{amsart}
\usepackage[T1]{fontenc}
\usepackage[utf8]{inputenc}
\usepackage{amsthm}
\usepackage{amssymb}

\makeatletter
\numberwithin{equation}{section}
\numberwithin{figure}{section}
  \theoremstyle{remark}
  \newtheorem*{acknowledgement*}{Acknowledgement}
  \theoremstyle{remark}
  \newtheorem*{rem*}{Remark}
\theoremstyle{plain}
\newtheorem{thm}{Theorem}
  \theoremstyle{definition}
  \newtheorem{defn}[thm]{Definition}
  \theoremstyle{plain}
  \newtheorem{prop}[thm]{Proposition}
 \theoremstyle{definition}
  \newtheorem{example}[thm]{Example}
  \theoremstyle{plain}
  \newtheorem{lem}[thm]{Lemma}
  \theoremstyle{plain}
  \newtheorem{cor}[thm]{Corollary}

\makeatother

\usepackage{babel}

\begin{document}

\title{On the Product in Negative Tate Cohomology for Finite Groups}

\author{Haggai Tene}
\begin{abstract}
Our aim in this paper is to give a geometric description of the cup
product in negative degrees of Tate cohomology of a finite group with
integral coefficients. By duality it corresponds to a product in the
integral homology of $BG$: {\normalsize \[
H_{n}(BG,\mathbb{Z})\otimes H_{m}(BG,\mathbb{Z})\rightarrow H_{n+m+1}(BG,\mathbb{Z})\]
}for $n,m>0$. We describe this product as join of cycles, which explains
the shift in dimensions. Our motivation came from the product defined
by Kreck using stratifold homology. We then prove that for finite
groups the cup product in negative Tate cohomology and the Kreck product
coincide. The Kreck product also applies to the case where $G$ is
a compact Lie group (with an additional dimension shift).
\end{abstract}
\maketitle

\section*{Introduction}

For a finite group $G$ one defines Tate cohomology with coefficients
in a $\mathbb{Z}[G]$ module $M$, denoted by $\widehat{H}^{*}(G,M)$.
This is a multiplicative theory: \[
\widehat{H}^{n}(G,M)\otimes\widehat{H}^{m}(G,M')\rightarrow\widehat{H}^{n+m}(G,M\otimes M')\]
and the product is called cup product. For $n>0$ there is a natural
isomorphism $H^{n}(G,M)\rightarrow\widehat{H}^{n}(G,M)$, and for
$n<-1$ there is a natural isomorphism $\widehat{H}^{n}(G,M)\rightarrow H_{-n-1}(G,M)$.
We restrict ourselves to coefficients in the trivial module $\mathbb{Z}$.
In this case, $\widehat{H}^{*}(G,\mathbb{Z})$ is a graded ring. Also,
in this case the group cohomology and homology are actually the cohomology
and homology of a topological space, namely $BG$, the classifying
space of principal $G$ bundles - $H^{n}(G,\mathbb{Z})\cong H^{n}(BG,\mathbb{Z})$
and $H_{n}(G,\mathbb{Z})\cong H_{n}(BG,\mathbb{Z})$. Combining this
with the isomorphism we had before $\widehat{H}^{n}(G,\mathbb{Z})\rightarrow H_{-n-1}(G,\mathbb{Z})$
for $n<-1$ we get a product $H_{n}(BG,\mathbb{Z})\otimes H_{m}(BG,\mathbb{Z})\rightarrow H_{n+m+1}(BG,\mathbb{Z})$
for $n,m>0$. Note the dimension shift. This product, with coefficients
in a field of characteristics $p$ rather than $\mathbb{Z}$, was
studied in \cite{key-2}. Our aim in this paper is to give a geometric
description of this product. We give a rather concrete description
in singular homology that involves the join of cycles, and that explains
the shift in dimension. 

Our motivation came from a geometric description of $H_{*}(G,\mathbb{Z})$
which appears in \cite{K2} and the product defined by Kreck using
stratifold homology. We then prove that the cup product in negative
Tate cohomology and the Kreck product coincide. An advantage in Kreck
theory is that it holds also for compact Lie groups giving a product:
\[
H_{n}(BG,\mathbb{Z})\otimes H_{m}(BG,\mathbb{Z})\rightarrow H_{n+m+1+dim(G)}(BG,\mathbb{Z})\]

\begin{acknowledgement*}
This paper is a part of the author's PhD thesis, written under the
direction of Prof. Matthias Kreck in the Hausdorff Research Institute
for Mathematics (HIM). The author would like to thank Prof. Kreck
for his support and the Hausdorff institute for the time he spent
there. 
\end{acknowledgement*}

\section*{Tate cohomology}
\begin{rem*}
In this paper $R$ is assumed to be a ring with unit, not necessarily
commutative, and all modules are assumed to be left $R$-modules unless
stated otherwise. The group $G$ is assumed to be finite unless stated
otherwise.
\end{rem*}
We start by defining Tate cohomology and the cup product as appears
in \cite{key-4}. To do so we introduce the language taken from the
stable module category. We will not get into details, for a formal
treatment the reader is referred to the appendix.

Let $M,N$ be two $R$-modules, denote by $\underline{Hom}_{R}(M,N)$
the quotient of $Hom_{R}(M,N)$ by the maps that factor through some
projective module. 
\begin{defn}
Given an $R$-module $M$, denote by $\Omega^{k}M$ the following
module:\\
Take any partial projective resolution of $M$, $P_{k-1}\xrightarrow{d_{k-1}}P_{k-2}...P_{0}\rightarrow M$
then $\Omega^{k}M=ker(P_{k-1}\xrightarrow{d_{k-1}}P_{k-2})$. If $k=1$
we simply denote it by $\Omega M$. This module clearly depends on
the choice of the resolution. Nevertheless, as proved in the appendix,
the modules $\underline{Hom}_{R}(\Omega^{k}M,\Omega^{l}N)$ do not
depend on the choice of resolutions i.e., they are well defined up
to canonical isomorphisms. If we would like to stress the dependency
on $P$ we would use the notation $\Omega_{P}^{k}M$. 
\end{defn}
\begin{flushleft}
Note that there is a natural map $\Psi:\underline{Hom}_{R}(M,N)\rightarrow\underline{Hom}_{R}(\Omega M,\Omega N)$. 
\par\end{flushleft}
\begin{defn}
The Tate cohomology of $G$ with coefficients in a $\mathbb{Z}[G]$
module $M$ is given by: \[
\widehat{H}^{n}(G,M)=\widehat{Ext}_{\mathbb{Z}[G]}^{n}(\mathbb{Z},M)={\textstyle \underset{m}{\underrightarrow{lim}}}\underline{Hom}_{\mathbb{Z}[G]}(\Omega^{n+m}\mathbb{Z},\Omega^{m}M)\]
where $\mathbb{Z}$ is the trivial $\mathbb{Z}[G]$-module (if $n<0$
we start this sequence from $m=-n$). 
\end{defn}
In our case, where $G$ is finite, we have the following proposition
which is proved in the appendix:
\begin{prop}
If $G$ is a finite group and $M$ is a $\mathbb{Z}[G]$-module which
is projective as a $\mathbb{Z}$-module then the homomorphism $\Psi:\underline{Hom}_{\mathbb{Z}[G]}(M,N)\rightarrow\underline{Hom}_{\mathbb{Z}[G]}(\Omega M,\Omega N)$
is an isomorphism. 
\end{prop}
Therefore, since $\mathbb{Z}$ and $\Omega^{k}\mathbb{Z}$ are projective
as $\mathbb{Z}$-modules this limit equals to \[
\widehat{H}^{n}(G,M)=\underline{Hom}_{\mathbb{Z}[G]}(\Omega^{n}\mathbb{Z},M)\]
if $n\geq0$ or \[
\widehat{H}^{n}(G,M)=\underline{Hom}_{\mathbb{Z}[G]}(\mathbb{Z},\Omega^{-n}M)\]
if $n<0$. Our main interest will be the second case, especially when
$M=\mathbb{Z}$.
\begin{example}
$\widehat{H}^{-1}(G,\mathbb{Z})=\underline{Hom}_{\mathbb{Z}[G]}(\mathbb{Z},\Omega\mathbb{Z})$.
Take the following exact sequence $0\rightarrow I\rightarrow\mathbb{Z}[G]\xrightarrow{f}\mathbb{Z}\rightarrow0$
where the map $f$ is the augmentation map and $I$ is the augmentation
ideal, so $I=\Omega\mathbb{Z}$. Therefore $Hom_{\mathbb{Z}[G]}(\mathbb{Z},\Omega\mathbb{Z})\cong Hom_{\mathbb{Z}[G]}(\mathbb{Z},I)=\left\{ 0\right\} $
so $\widehat{H}^{-1}(G,\mathbb{Z})=\left\{ 0\right\} $.

Let $G$ be a finite group. We construct a natural isomorphism $\widehat{H}^{-n-1}(G,\mathbb{Z})\rightarrow H_{n}(G,\mathbb{Z})$
for $n\geq1$. Before that we prove a small lemma.\end{example}
\begin{lem}
Let $G$ be a finite group and $P$ a projective $\mathbb{Z}[G]$-module,
then for every element $x\in P$ we have:\\
1) $x\in P^{G}\Leftrightarrow\exists y\in P,x=Ny$\\
2) $y\otimes1=y'\otimes1\in P\otimes_{\mathbb{Z}[G]}\mathbb{Z}\Leftrightarrow Ny=Ny'$\\
Where $P^{G}$ are the invariants of $P$ under the action of $G$,
$N$ is the norm homomorphism defined by multiplication by the element
$N=\underset{g\in G}{\sum}g\in\mathbb{Z}[G]$.\end{lem}
\begin{proof}
For every $\mathbb{Z}[G]$-module $M$ the following sequence is exact:\[
0\rightarrow\widehat{H}^{-1}(G,M)\rightarrow H_{0}(G,M)\rightarrow H^{0}(G,M)\rightarrow\widehat{H}^{0}(G,M)\rightarrow0\]
where the map $H_{0}(G,M)\rightarrow H^{0}(G,M)$ is the norm map
$N:M\otimes\mathbb{Z}\rightarrow M^{G}$ given by ($N(x\otimes k)=kNx$)
(\cite{key-3} VI,4). If $M$ is projective then $\widehat{H}^{m}(G,M)=0$
for all $m\in\mathbb{Z}$, hence $N$ is an isomorphism. We conclude:

\begin{flushleft}
1) Surjectivity of $N$ implies that $x\in P^{G}\Leftrightarrow\exists y\in P,x=Ny$.\\
2) Injectivity of $N$ implies that $y\otimes1=y'\otimes1$ $\Leftrightarrow$
$Ny=Ny'$ for all $y,y'\in P$ .
\par\end{flushleft}\end{proof}
\begin{prop}
\label{pro:Let-G-be}Let G be a finite group then there is an isomorphism
between $\widehat{H}^{-n-1}(G,\mathbb{Z})$ and $H_{n}(G,\mathbb{Z})$
for $n\geq1$.\end{prop}
\begin{proof}
Take a projective resolution of $\mathbb{Z}$ - $\cdots\rightarrow P_{n}\xrightarrow{d_{n}}P_{n-1}...\rightarrow P_{0}\rightarrow\mathbb{Z}$,
taking the tensor of it with $\mathbb{Z}$ gives us the chain complex
for the homology of $G$ which we denote by $C_{*}(G)$. We define
a map from $Hom_{\mathbb{Z}[G]}(\mathbb{Z},\Omega^{n+1}\mathbb{Z})$
to $C_{n}(G)$ the following way: Given a homomorphism $f:\mathbb{Z}\rightarrow\Omega^{n+1}\mathbb{Z}$,
$f(1)=x$ is an invariant element in $P_{n}$. By the lemma, since
$P_{n}$ is projective and $x$ is invariant there is some $y\in P_{n}$
such that $x=Ny$. We define $\Phi(f)=y\otimes1$. This doesn't depend
on the choice of $y$ since $Ny=Ny'\Leftrightarrow y\otimes1=y'\otimes1$
by the lemma above. We know that $Nd_{n}(y)=d_{n}(Ny)=d_{n}(x)=0$
and by the lemma this implies that $d_{n}(y)\otimes1=0$ ($P_{n-1}$
is projective and here we use the fact that $n\geq1$). We deduce
that $y\otimes1\in Z_{n}(G)$. The map described now $Hom_{\mathbb{Z}[G]}(\mathbb{Z},\Omega^{n+1}\mathbb{Z})\rightarrow Z_{n}(G)$
is surjective since given an element $y\otimes1\in C_{n}(G)$ such
that $d_{n}(y)\otimes1=0$ this implies that $Nd_{n}(y)=0$, so we
define $f(k)=kNy$, this is well defined since $Ny$ is invariant
and in the kernel of $d_{n}$. 

We now have a surjective homomorphism $\Phi:Hom_{\mathbb{Z}[G]}(\mathbb{Z},\Omega^{n+1}\mathbb{Z})\rightarrow H_{n}(G,\mathbb{Z})$.
If $f\in ker(\Phi)$ then there exist $s\in P_{n+1}$ such that $\Phi(f)=y\otimes1=d_{n+1}(s)\otimes1$
then the map $f:\mathbb{Z}\rightarrow\Omega^{n+1}\mathbb{Z}$ factors
through $P_{n+1}$ which is projective by $1\mapsto Ns$. On the other
hand if $f$ factors through a projective module, w.l.o.g. $P_{n+1}$,
then $Ny=f(1)=d_{n+1}(Ns)$ (every invariant element in $P_{n+1}$
is of the form $Ns$ by the lemma). This implies that $Nd_{n+1}(s)=Ny\Leftrightarrow d_{n+1}(s\otimes1)=d_{n+1}(s)\otimes1=y\otimes1$.

We conclude that the induced map: {\small \[
\underline{\Phi}:\widehat{H}^{-n-1}(G,\mathbb{Z})=\underline{Hom}_{\mathbb{Z}[G]}(\mathbb{Z},\Omega^{n+1}\mathbb{Z})\rightarrow H_{n}(G,\mathbb{Z})\]
}is an isomorphism for all $n\geq1$.\end{proof}
\begin{rem*}
Since $\widehat{H}^{-1}(G,\mathbb{Z})=\left\{ 0\right\} $ we conclude
that $\widehat{H}^{-n-1}(G,\mathbb{Z})\cong\tilde{H}_{n}(G,\mathbb{Z})$
for $n\geq0$ where $\tilde{H}_{n}(G,\mathbb{Z})$ is the reduced
homology.
\end{rem*}

\section*{The product structure}

The cup product in Tate cohomology $\widehat{H}^{-n}(G,\mathbb{Z})\otimes\widehat{H}^{-m}(G,\mathbb{Z})\rightarrow\widehat{H}^{-n-m}(G,\mathbb{Z})$
is given by composition (this is also called the Yoneda composition
product). Given: \[
[f]\in\widehat{H}^{-n}(G,\mathbb{Z})=\underline{Hom}_{\mathbb{Z}[G]}(\mathbb{Z},\Omega^{n}\mathbb{Z})\]
\[
[g]\in\widehat{H}^{-m}(G,\mathbb{Z})=\underline{Hom}_{\mathbb{Z}[G]}(\mathbb{Z},\Omega^{m}\mathbb{Z})\cong\underline{Hom}_{\mathbb{Z}[G]}(\Omega^{n}\mathbb{Z},\Omega^{n+m}\mathbb{Z})\]
we compose them to get a map: \[
[f]\cup[g]=[g\circ f]\in\underline{Hom}_{\mathbb{Z}[G]}(\mathbb{Z},\Omega^{n+m}\mathbb{Z})\]
Since for $n,m\geq2$ we have $\widehat{H}^{-n}(G,\mathbb{Z})\cong H_{n-1}(G,\mathbb{Z}),\widehat{H}^{-m}(G,\mathbb{Z})\cong H_{m-1}(G,\mathbb{Z})$
we have a product $H_{n-1}(G,\mathbb{Z})\otimes H_{m-1}(G,\mathbb{Z})\rightarrow H_{n+m-1}(G,\mathbb{Z})$.
We would like to have a description of the isomorphism $\underline{Hom}_{\mathbb{Z}[G]}(\mathbb{Z},\Omega^{m}\mathbb{Z})\cong\underline{Hom}_{\mathbb{Z}[G]}(\Omega^{n}\mathbb{Z},\Omega^{n+m}\mathbb{Z})$
which is concrete. To do so we use the following construction:

\subsection*{The join of augmented chain complexes}

~

Let $G$ be a finite group and let $P$ and $Q$ be the following
augmented chain complexes over $\mathbb{Z}[G]$ - $...\rightarrow P_{2}\rightarrow P_{1}\rightarrow P_{0}\rightarrow\mathbb{Z}$
and $...\rightarrow Q_{2}\rightarrow Q_{1}\rightarrow Q_{0}\rightarrow\mathbb{Z}$.
We define the join of those two chain complexes to be $P\ast Q=\Sigma(P\otimes_{\mathbb{Z}}Q)$
that is the suspension of the tensor product over $\mathbb{Z}$ (with
a diagonal $G$ action). To be more specific $(P\ast Q)_{n}=\underset{0\leq k\leq n+1}{\oplus}P_{k-1}\otimes_{\mathbb{Z}}Q_{n-k}$:

\[
...\to P_{1}\otimes_{\mathbb{Z}}\mathbb{Z}\oplus P_{0}\otimes_{\mathbb{Z}}Q_{0}\oplus\mathbb{Z}\otimes_{\mathbb{Z}}Q_{1}\to P_{0}\otimes_{\mathbb{Z}}\mathbb{Z}\oplus\mathbb{Z}\otimes_{\mathbb{Z}}Q_{0}\to\mathbb{Z}\otimes_{\mathbb{Z}}\mathbb{Z}\]
$P\ast Q$ is an augmented $\mathbb{Z}[G]$ chain complex in a natural
way. 
\begin{lem}
If both $P$ and $Q$ are projective and acyclic augmented $\mathbb{Z}[G]$
chain complexes then $P\ast Q$ is a projective and acyclic augmented
$\mathbb{Z}[G]$ chain complex.\end{lem}
\begin{proof}
$P$ and $Q$ are projective acyclic chain complexes over $\mathbb{Z}$
so the same is true for their tensor product, by the $K\ddot{u}nneth$
formula. $(P\ast Q)_{n}$ is projective over $\mathbb{Z}[G]$ for
$n\geq0$ since each of the modules $P_{k-1}\otimes_{\mathbb{Z}}Q_{n-k}$
is projective.\end{proof}
\begin{lem}
Let $P$ and $Q$ be \textup{two resolutions of }$\mathbb{Z}$ over
$\mathbb{Z}[G]$, and let $s\in Q_{n-1}$ be an element, $n>1$. Define
a map $s_{*}:P_{k-1}\rightarrow(P\ast Q)_{k+n-1}$ by $s_{\ast}(x)=x\otimes s$
called the join with $s$. Then we have:\\
1) $s_{*}$ is a group homomorphism.\\
2) If $s$ is $G$-invariant then $s_{*}$ is a homomorphism over
$\mathbb{Z}[G]$.\\
3) If $s\in ker(Q_{n-1}\rightarrow Q_{n-2})$ then $s_{*}$ is
a chain map of degree $n$.\end{lem}
\begin{proof}
1) Follows from the properties of the tensor product. \\
2) For every $g\in G$ we have: $g(s_{*}(x))=g(x\otimes s)=g(x)\otimes g(s)=g(x)\otimes s=s_{*}(g(x))$\\
3) $\partial(s_{*}(x))=\partial(x\otimes s)=\partial(x)\otimes s+(-1)^{\left|x\right|+1}x\otimes\partial s=\partial(x)\otimes s=s_{*}(\partial(x))$
\end{proof}
This implies the following:
\begin{thm}
Let $n,m>0$, the product $\widehat{H}^{-n}(G,\mathbb{Z})\otimes\widehat{H}^{-m}(G,\mathbb{Z})\rightarrow\widehat{H}^{-n-m}(G,\mathbb{Z})$
is given by $[f]\cup[g]=[f\ast g]$ where $(f\ast g)(k)=k\cdot f(1)\otimes g(1)\in\Omega_{P\ast P}^{m+n}\mathbb{Z}$
and $k\in\mathbb{Z}$.\end{thm}
\begin{proof}
Take a projective resolution $P$ for $\mathbb{Z}$ over $\mathbb{Z}[G]$.
Let $[f]\in\widehat{H}^{-n}(G,\mathbb{Z})=\underline{Hom}_{\mathbb{Z}[G]}(\mathbb{Z},\Omega^{n}\mathbb{Z})$,
$[g]\in\widehat{H}^{-m}(G,\mathbb{Z})=\underline{Hom}_{\mathbb{Z}[G]}(\mathbb{Z},\Omega^{m}\mathbb{Z})\cong\underline{Hom}_{\mathbb{Z}[G]}(\Omega^{n}\mathbb{Z},\Omega^{n+m}\mathbb{Z})$.
Choose representatives $f,g$ and define a degree $m$ map $P\rightarrow P\ast P$
by $x\mapsto x\otimes g(1)$. Since $g(1)$ is invariant and in the
kernel this map is a chain map of $\mathbb{Z}[G]$ chain complexes
of degree $m$. This gives us a concrete construction of the isomorphism
$\underline{Hom}_{\mathbb{Z}[G]}(\mathbb{Z},\Omega^{m}\mathbb{Z})\cong\underline{Hom}_{\mathbb{Z}[G]}(\Omega^{n}\mathbb{Z},\Omega^{n+m}\mathbb{Z})$.
The composition is therefore $g\circ f(1)=f(1)\otimes g(1)$.
\end{proof}

\subsection*{A description of the product by joins of cycles}

~

We now consider resolutions which come from singular chains of spaces.
Let $G$ be a finite group, recall that a contractible $G-CW$ complex
with a free $G$ action is denoted by $EG$, the quotient space $EG/G$
is the classifying space of principal $G$ bundles and is denoted
by $BG$.

We consider now the augmented singular chain complex of $EG$ denoted
by $C_{*}(EG)$. The action on $EG$ is free so $C_{*}(EG)$ is projective
($n\geq0$) and $EG$ is contractible so $C_{*}(EG)$ is acyclic.

As we saw before (prop. \ref{pro:Let-G-be}) every element of $H_{n}(G,\mathbb{Z})$
can be considered as an invariant cycle in $C_{n}(EG)$ (modulo invariant
boundary), we will show that the product can be considered as the
join of the two such cycles, which is naturally an invariant cycle
in $C_{*}(EG\ast EG)$ where $EG\ast EG$ is the join of two copies
of $EG$. Note that since the join of contractible spaces is contractible,
$EG\ast EG$ is contractible, and it has a natural $G$ action, given
by $g(x,y,t)=(gx,gy,t)$, which is free since it is free on both copies
of $EG$. This implies that its augmented singular chain complex is
a projective resolution of $\mathbb{Z}$ over $\mathbb{Z}[G]$.

We now associate the join of chain complexes to the join of spaces.
\begin{lem}
Let $A$ and $B$ be two spaces and let $C_{*}(A)$ and $C_{*}(B)$
be their augmented (!) singular chain complexes, then there is a natural
chain map: \[
h:C_{*}(A)\ast C_{*}(B)\rightarrow C_{*}(A\ast B)\]
If $G$ acts on $A$ and $B$ then it also acts on $A\ast B$ and
the chain complexes are complexes over $\mathbb{Z}[G]$ and $h$ is
a map of $\mathbb{Z}[G]$ chain complexes.\end{lem}
\begin{proof}
We first note that for $n,m\geq0$, for every two singular simplices
$\sigma\in C_{n}(A)$ and $\tau\in C_{m}(B)$ there is a canonical
singular chain $\sigma\ast\tau\in C_{n+m+1}(A\ast B)$ and this definition
can be extended in a bilinear way to chains. Define $h$ the following
way:\\
Given an element $s\otimes t\in C_{n}(A)\otimes C_{m}(B)$, if
$n,m\geq0$ then $h(s\otimes t)=s\ast t$, else $n=-1$ (or $m=-1$)
then $s$ is an integer, denote it by $k$ then $h(s\otimes t)=h(k\otimes t)=k\cdot t$
where $t$ is the chain induced by the inclusion of $B$ in $A\ast B$
(and similarly for $m=-1$).

We have to show that $h$ is a chain map. For two simplices of positive
(!) dimension we have the formula $\partial(\sigma\ast\tau)=\partial(\sigma)\ast\tau+(-1)^{dim(\sigma)+1}\sigma\ast\partial(\tau)$.
The formula extends to chains, so we have:\\
{\small $\partial h(s\otimes t)=\partial(s\ast t)=\partial(s)\ast t+(-1)^{\left|s\right|+1}s\ast\partial(t)=h(\partial(s)\otimes t+(-1)^{\left|s\right|+1}s\otimes\partial(t))=h(\partial(s\otimes t))$.}\\
For $\sigma$, a simplex of dimension $0$ (a point), $\sigma\ast\tau$
is the cone over $\tau$ and its boundary is given by $\partial(\sigma\ast\tau)=\tau+(-1)^{dim(\sigma)+1}\sigma\ast\partial(\tau)$.
Since the boundary map $C_{0}(A)\rightarrow\mathbb{Z}$ is the augmentation
map we see indeed that also in this case $h$ commutes with the boundary
(with respect to the way we have defined $h(k\otimes t)$).

The boundary formula is not (!) true when one of the simplices is
zero dimensional due to the non symmetric way we define the faces
of a zero simplex (the $n$ simplex has $n+1$ faces while the zero
simplex has no faces). If we wanted to be consistent with the boundaries
of the higher simplices we should have used only augmented chain complexes.
More in this direction appears in \cite{key-6}.

When there is a $G$ action on both spaces then clearly all the complexes
are complexes over $\mathbb{Z}[G]$. $h$ is a $\mathbb{Z}[G]$ chain
map since for every $g\in G$ we have: \[
h(g(s\otimes t))=h(gs\otimes gt))=gs\ast gt=g(s\ast t)=g(h(s\otimes t))\]
\end{proof}
\begin{thm}
The cup product in negative Tate cohomology gives a product \[
H_{n}(G,\mathbb{Z})\otimes H_{m}(G,\mathbb{Z})\rightarrow H_{n+m+1}(G,\mathbb{Z})\]
 ($n,m>0$). Each homology class in $H_{n}(G,\mathbb{Z})$ is represented
by an invariant cycle in $EG$. The product of two classes is given
by the join of those cycles, which is an invariant cycle in $EG\ast EG$.\end{thm}
\begin{proof}
We already saw that the product can be described by the join of resolutions.
By the proposition above there is a degree zero chain map $C_{*}(EG)\ast C_{*}(EG)\rightarrow C_{*}(EG\ast EG)$.
The image of $f(1)\otimes g(1)$ under this map is the join of $f(1)$
with $g(1)$. This gives a more concrete model where the cycles are
actual invariant singular cycles of the space $EG\ast EG$.\end{proof}
\begin{cor}
The product in $\tilde{H}_{*}(G,\mathbb{Z})$ comes from the chain
map: \[
\left(P\otimes_{\mathbb{Z}[G]}\mathbb{Z}\right)\otimes_{\mathbb{Z}}\left(P\otimes_{\mathbb{Z}[G]}\mathbb{Z}\right)\to\left(P*P\right)\otimes_{\mathbb{Z}[G]}\mathbb{Z}\]
given by: $(x\otimes1)\otimes(y\otimes1)\to\left((Nx)\otimes y\right)\otimes1$,
where $P$ is an augmented projective resolution.
\end{cor}
This map is equal to the composition of two maps. The first one:

\[
\left(P\otimes_{\mathbb{Z}[G]}\mathbb{Z}\right)\otimes_{\mathbb{Z}}\left(P\otimes_{\mathbb{Z}[G]}\mathbb{Z}\right)\to\left(P*P\right)\otimes_{\mathbb{Z}[G\times G]}\mathbb{Z}\]
is given by $(x\otimes1)\otimes(y\otimes1)\to\left(x\otimes y\right)\otimes1$.
This is an exterior product, which is injective (in homology) by the
$K\ddot{u}nneth$ theorem. Note that the homology of $\left(P*P\right)\otimes_{\mathbb{Z}[G\times G]}\mathbb{Z}$
need not be equal to $\tilde{H}_{*}(G\times G,\mathbb{Z})$ since
$P*P$ is not projective over $\mathbb{Z}[G\times G]$. $\left(P*P\right)\otimes_{\mathbb{Z}[G\times G]}\mathbb{Z}$
is the chain complex of $BG*BG$, and this is the join product: \[
\tilde{H}_{n}(BG,\mathbb{Z})\otimes\tilde{H}_{m}(BG,\mathbb{Z})\to\tilde{H}_{m+n+1}(BG*BG,\mathbb{Z})\]
The second map, $\left(P*P\right)\otimes_{\mathbb{Z}[G\times G]}\mathbb{Z}\to\left(P*P\right)\otimes_{\mathbb{Z}[G]}\mathbb{Z}$
is given by $\left(x\otimes y\right)\otimes1\to\left((Nx)\otimes y\right)\otimes1$,
which is a transfer map.

\section*{Comparing Kreck's product and the cup product in Tate cohomology}

The Kreck product is defined using stratifolds and stratifold homology.
Stratifolds are generalization of manifolds. They were introduced
by Kreck \cite{K} and used in order to define a bordism theory, denoted
by $SH_{*}$, which is naturally isomorphic to singular homology.
We will use them to describe group homology with integral coefficients
and the Kreck product.

\subsection*{Stratifolds}

~

Kreck defined stratifolds as spaces with a sheaf of functions, called
the smooth functions, fulfilling certain properties but for our purpose
the following definition is enough (these stratifolds are also called
p-stratifolds):

A stratifold is a pair consisting of a topological space and a subsheaf
of the sheaf of real continuous functions, which is constructed inductively
in a similar way to the way we construct $CW$ complexes. We start
with a discrete set of points denoted by $X^{0}$ and define inductively
the set of smooth functions which in the case of $X^{0}$ are all
real functions. 

Suppose $X^{k-1}$ together with a smooth set of functions is given.
Let $W$ be a $n$ dimensional smooth manifold {}``the $n$ strata''
with boundary and a collar $c$, and $f$ a continuous map from the
boundary of $W$ to $X^{n-1}$. We require that $f$ is smooth, which
means that its composition with every smooth map from $X^{n-1}$ is
smooth. Define $X^{n}=X^{n-1}\cup_{f}W$. The smooth maps on $X^{n}$
are defined to be those maps $g:X^{n}\to\mathbb{R}$ which are smooth
when restricted to $X^{n-1}$ and to $W$ and such that for some $0<\delta$
we have $gc(x,t)=gf(x)$ for all $x\in\partial W$ and $t<\delta$.

Among the examples of stratifolds are manifolds, real and complex
algebraic varieties \cite{Gri}, and the one point compactification
of a smooth manifold. The cone over a stratifold and the product of
two stratifolds are again stratifolds. 

We can also define stratifolds with boundary which are analogous to
manifolds with boundary. A main difference is that every stratifold
is the boundary of its cone, which is a stratifold with boundary. 

Given two stratifolds with boundary $(T',S')$ and $(T'',S'')$ and
an isomorphism $f:S'\rightarrow S''$ there is a well defined stratifold
structure on the space $T'\cup_{f}T''$ which is called the gluing.
On the other hand, given a smooth map $g:T\rightarrow\mathbb{R}$
such that there is a neighborhood of $0$ which consists only of regular
values then the preimages $g^{-1}((-\infty,0])=T'$ and $g^{-1}([0,\infty))=T''$
are stratifolds with boundary and $T$ is isomorphic to the gluing
$T'\cup_{Id}T''$.

To obtain singular homology we specialize our stratifolds in the following
way: We use compact stratifolds, require that their top stratum will
be oriented and the codimension one stratum will be empty.
\begin{rem*}
Regarding regularity, a condition often required, see \cite{K3}.
\end{rem*}

\subsection*{Stratifold homology}

~

Stratifold homology was defined by Kreck in \cite{K}. We will describe
here a variant of this theory called parametrized stratifold homology,
which is naturally isomorphic to it for $CW$ complexes. In this paper
we will refer to parametrized stratifold homology just as stratifold
homology and use the same notation for it. 

Stratifold homology is a homology theory, denoted by $SH_{*}$. It
is naturally isomorphic to integral homology and gives a new geometric
point of view on it.
\begin{defn}
Let $X$ be a topological space and $n\geq0$, define $SH_{n}(X)$
to be $\left\{ g:S\rightarrow X\right\} /\sim$ i.e., bordism classes
of maps $g:S\rightarrow X$ where $S$ is a compact oriented stratifold
of dimension $n$ and $g$ is a continuous map. We often denote the
class $[g:S\rightarrow X]$ by $[S,g]$ or by $[S\rightarrow X]$.
$SH_{n}(X)$ has a natural structure of an Abelian group, where addition
is given by disjoint union of maps and the inverse is given by reversing
the orientation. If $f:X\rightarrow Y$ is a continuous map than we
can define an induced map by composition $f_{*}:SH_{n}(X)\rightarrow SH_{n}(Y)$. 
\end{defn}
One constructs a boundary operator and prove the following:
\begin{thm}
(Mayer-Vietoris) The following sequence is exact:\[
...\rightarrow SH_{n}(U\cap V)\rightarrow SH_{n}(U)\oplus SH_{n}(V)\rightarrow SH_{n}(U\cup V)\xrightarrow{\partial}SH_{n-1}(U\cap V)\rightarrow...\]
where, as usual, the first map is induced by inclusions and the second
is the difference of the maps induced by inclusions. 
\end{thm}
$SH_{*}$ with the boundary operator is a homology theory. 
\begin{thm}
There is a natural isomorphism of homology theories $\Phi:SH_{*}\to H_{*}$\end{thm}
\begin{proof}
See for example \cite{key-12}. $\Phi$ is given by $\Phi_{n}([S,f])=f_{*}([S])$
where $[S]\in H_{n}(S,\mathbb{Z})$ is the fundamental class of $S$.
\end{proof}

\subsection*{Stratifold group homology}

~

One defines the group homology of a group $G$ with coefficients in
a $\mathbb{Z}[G]$ module $M$ to be $H_{*}(BG,M)$ where $M$ is
considered as a local coefficients system. Our main interest is when
$M=\mathbb{Z}$ with the trivial action, then this reduces to the
integral homology $H_{*}(BG,\mathbb{Z})$. These groups are naturally
isomorphic to the groups $SH_{*}(BG,\mathbb{Z})$ by the theorem above. 

Let $G$ be a compact Lie group of dimension $d$. Denote by $SH_{n}(G,\mathbb{Z})$
the set of compact oriented stratifolds of dimension $n$ with a free
and orientation preserving $G$ action modulo $G$-cobordism, i.e.
a cobordism with a free $G$ action extending the given action on
the boundary (all actions on the stratifolds are assumed to be smooth).
We denote the class of the stratifold and the action by $[S,\rho]$. 

The following lemma and proposition are an easy exercise:
\begin{lem}
1) Let $S$ be a compact oriented  stratifold of dimension $n$ and
$\widetilde{S}\rightarrow S$ a covering space then $\widetilde{S}$
can be given a unique structure of an oriented stratifold such that
the covering map is an orientation preserving local isomorphism. If
$S$ is compact and the fibers are finite then $\widetilde{S}$ is
compact. 

2) Let $S$ be a compact oriented  stratifold of dimension $n$ with
an orientation preserving free $G$ action then $S/G$ can be given
a unique structure of a compact oriented stratifold such that  the
projection will be an orientation preserving local isomorphism.\end{lem}
\begin{prop}
Let $G$ be a finite group, the map $\Psi:SH_{n}(G,\mathbb{Z})\rightarrow SH_{n}(BG,\mathbb{Z})$
given be $[S,\rho]\mapsto[S/G\xrightarrow{f}BG]$, where $f$ is the
classifying map, is an isomorphism.\\
The map $\Psi^{-1}:SH_{n}(BG,\mathbb{Z})\rightarrow SH_{n}(G,\mathbb{Z})$
is given by $[S\xrightarrow{f}BG]\mapsto[\widetilde{S},\rho]$ where
$\widetilde{S}\rightarrow S$ is the pull back of the universal bundle
$EG\rightarrow BG$ and $\rho$ is the induced action.\end{prop}
\begin{rem*}
Similarly, an isomorphism $\Psi:SH_{n+d}(G,\mathbb{Z})\rightarrow SH_{n}(BG,\mathbb{Z})$
can be constructed for a compact Lie group $G$ of dimension $d$.
\end{rem*}
There is a natural product structure $SH_{n}(G,\mathbb{Z})\otimes SH_{m}(G,\mathbb{Z})\rightarrow SH_{n+m}(G,\mathbb{Z})$
given by the Cartesian product with the diagonal action: \[
[S,\rho]\otimes[S',\rho']\rightarrow[S\times S',\Delta]\]
This product vanishes whenever $n,m>0$ since it is the boundary of
$[CS\times S',\tilde{\rho}]$ where $\tilde{\rho}$ is the obvious
extension of the action $\Delta$, but it is also the boundary of
$[S\times CS',\hat{\rho}]$ where $\hat{\rho}$ is the obvious extension
of the action $\Delta$. 

The Kreck product is a secondary product defined by gluing $(CS\times S',\tilde{\rho})$
and $(S\times CS',\hat{\rho})$ along their common boundary $(S\times S',\Delta)$
\[
[S,\rho]\otimes[S',\rho']\rightarrow[S\ast S',\rho\ast\rho']\]
(note that after the gluing what we get is the join of the two stratifolds). 

The product $SH_{n}(G,\mathbb{Z})\otimes SH_{m}(G,\mathbb{Z})\rightarrow SH_{n+m+1}(G,\mathbb{Z})$
does not vanish in general. When $G$ is finite cyclic then $SH_{n}(G,\mathbb{Z})$
is infinite cyclic when $n=0$, zero when $n$ is even and isomorphic
to $G$ when $n$ is odd. The generators can be taken to be odd dimensional
spheres with the action induced by the complex multiplication, when
the sphere is considered as the unit sphere in a complex space. In
this case the product of generators is again a generator. A similar
construction will hold for $G=S^{1}$ and $S^{3}$. This implies that
the product is non trivial for every group with a free and orientation
preserving smooth action on a sphere.

There is an isomorphism $\Psi:SH_{n}(G,\mathbb{Z})\rightarrow\widehat{H}^{-n-1}(G,\mathbb{Z})$
{\small for }$n>0$ given by the composition:

\[
SH_{n}(G,\mathbb{Z})\rightarrow SH_{n}(BG,\mathbb{Z})\rightarrow H_{n}(BG,\mathbb{Z})\rightarrow\widehat{H}^{-n-1}(G,\mathbb{Z})\]

One might show that this isomorphism is given the following way: Take
some model for $EG$. Its singular chain complex $C_{*}(EG)$ is a
projective resolution for $\mathbb{Z}$ over $\mathbb{Z}[G]$. Let
$[(S,\rho)]$ be an element in $SH_{n}(G,\mathbb{Z})$. There is a
map $f:S\rightarrow EG$ that commutes with the action of $G$. This
map is unique up to $G$ homotopy ($f$ is called the classifying
map), any two such maps are $G$ homotopic. Since $f$ commute with
the action of $G$ it induces a map of the singular chain complexes
which are complexes of $\mathbb{Z}[G]$ modules - $C(S)\xrightarrow{f_{*}}C(EG)$.
As shown in \cite{key-12}, $S$ has a fundamental class, we take
some representative of it which is $G$ invariant (we can do that
by lifting a fundamental cycle of $S/G$) and denote it by $s$. We
get an element $f_{*}(s)\in C(EG)_{n}$ which is both invariant and
a cycle thus we get an element in $Hom(\mathbb{Z},\Omega^{n+1})$.
As before different choices of $S$ and $f$ will give elements that
differ by a map which factors through a projective (the fundamental
class of the cobordism is mapped into $C(EG)_{n+1}$ which is projective),
hence gives a homomorphism $SH_{n}(G,\mathbb{Z})\rightarrow\underline{Hom}_{\mathbb{Z}[G]}(\mathbb{Z},\Omega^{n+1}\mathbb{Z})=\widehat{H}^{-n-1}(G,\mathbb{Z})$
which is exactly the isomorphism above.

Now we would like to show that the Kreck product is the same as the
cup product. We show that the join of two fundamental classes is equal
to the fundamental class of their join.
\begin{lem}
Let $S$ and $S'$ be two compact  oriented stratifolds of dimension
$n$ and $m$ ($m,n>0$) respectively. Denote the fundamental classes
of $S,S',S\ast S'$ by $a_{S},a_{S'},a_{S\ast S'}$ then $a_{S}\ast a_{S'}=a_{S\ast S'}$.\end{lem}
\begin{proof}
Let $U=\left\{ (s,s',t)\in S\ast S'\mid t<1\right\} $, $V=\left\{ (s,s',t)\in S\ast S'\mid0<t\right\} $
then $U\simeq S,V\simeq S',U\cap V\simeq S\times S',U\cup V=S\ast S'$.
By $Mayer\, Vietoris$ the boundary map $\partial:H_{n+m+1}(S\ast S',\mathbb{Z})\rightarrow H_{n+m}(S\times S',\mathbb{Z})$
is injective (an isomorphism actually) and by the definition of the
boundary we have $\partial(a_{S\ast S'})=a_{S\times S'}$. It will
be enough to show that $\partial(a_{S}\ast a_{S'})=a_{S\times S'}$.
We do know that $\partial(a_{S}\ast a_{S'})=a_{S}\times a_{S'}$,
this follows from the definition of the boundary after taking the
suitable representative for $a_{S}\ast a_{S'}$. So we reduced the
problem to proving that $a_{S}\times a_{S'}=a_{S\times S'}$. This
fact follows from the following commutative diagram:\\
\[
\begin{array}{ccc}
H_{n}(S,\mathbb{Z})\otimes H_{m}(S',\mathbb{Z}) & \xrightarrow{\times} & H_{n+m}(S\times S',\mathbb{Z})\\
\downarrow &  & \downarrow\\
H_{n}(S\mid s,\mathbb{Z})\otimes H_{m}(S'\mid s',\mathbb{Z}) & \xrightarrow{\times} & H_{n+m}(S\times S'\mid s\times s',\mathbb{Z})\end{array}\]
Where $H_{k}(X\mid x,\mathbb{Z})$ stands for $H_{k}(X,X/\left\{ x\right\} ,\mathbb{Z})$.
In order to show that $a_{S}\times a_{S'}=a_{S\times S'}$ we have
to show that for every $(s,s')\in S\times S'$ $a_{S}\times a_{S'}$
is mapped by the right vertical map to the generator of $H_{n+m}(S\times S'\mid s\times s')$.
By definition $a_{S}\otimes a_{S'}$ is mapped by the left vertical
map to the tensor of the generators for $H_{n}(S\mid s,\mathbb{Z})\otimes H_{m}(S'\mid s',\mathbb{Z})$.
Also by definition the element $a_{S}\otimes a_{S'}$ is mapped by
the upper arrow to $a_{S}\times a_{S'}$. By the commutativity of
the diagram it is enough to show that the lower arrow maps the tensor
of generators to the generator of the product. By excision, this can
be rephrased that the same holds for the map $H_{n}(\mathbb{R}^{n}\mid0,\mathbb{Z})\otimes H_{m}(\mathbb{R}^{m}\mid0,\mathbb{Z})\xrightarrow{\times}H_{n+m}(\mathbb{R}^{n+m}\mid0,\mathbb{Z})$
which is the case if we make the right choice of orientations.
\end{proof}
We have thus proved the following theorem:
\begin{thm}
Let $G$ be a finite group, then there is a natural isomorphism between
$SH_{n}(G,\mathbb{Z})$ and $\widehat{H}^{-n-1}(G,\mathbb{Z})$ and
this isomorphism respects the product.
\end{thm}
In other words, the product in group homology defined by Kreck using
stratifold homology and the join agrees with the cup product in negative
Tate cohomology.

\section*{Appendix - The stable module category}

In this appendix we give the background needed for the construction
we used for Tate cohomology.

Again $R$ is a ring with unit, not necessarily commutative, and all
modules are assumed to be left $R$-modules.

\subsection*{The stable category $St-mod(R)$}

~

Let $M$ and $N$ be two $R$-modules, denote by $PHom_{R}(M,N)$
the set of $R$-homomorphisms $M\xrightarrow{f}N$ that factors through
a projective $R$-module, i.e. there exists a projective $R$-module
$P$ and two maps $M\xrightarrow{f_{1}}P\xrightarrow{f_{2}}N$ such
that  $f=f_{2}\circ f_{1}$. The following proposition is left as
an easy exercise:
\begin{prop}
$PHom_{R}(M,N)$ is a sub module of $Hom_{R}(M,N)$ and the composition
of two homomorphisms such that one of them factors through a projective
module also factors through a projective module.
\end{prop}
\noindent By the proposition above we can define {\small $\underline{Hom}_{R}(M,N)=Hom_{R}(M,N)/PHom_{R}(M,N)$}
which is an $R$-module, and a composition {\small $\underline{Hom}_{R}(N,K)\times\underline{Hom}_{R}(M,N)\rightarrow\underline{Hom}_{R}(M,K)$}
which is $R$-bilinear. 
\begin{defn}
Let $R$ be a ring, denote by $St-mod(R)$ the category whose objects
are all $R$-modules and the morphisms between each $M$ and $N$
are $\underline{Hom}_{R}(M,N)$. This category is called the stable
module category. 
\end{defn}

\subsection*{The functor $\Omega$}

~

For every $R$-module $M$ choose (once and for all) a projective
cover, that is a surjective map $\pi_{M}:P_{M}\rightarrow M$ where
$P_{M}$ is a projective $R$-module (for example the canonical free
cover). 

Define a functor $\Omega:St-mod(R)\rightarrow St-mod(R)$ the following
way: For an object $M$ define $\Omega(M)=Ker(\pi_{M})$. For a morphism
$[f]\in\underline{Hom}_{R}(M,N)$ choose some representative $f:M\rightarrow N$,
use the fact that $P_{M}$ is projective and $\pi_{N}$ is surjective
to define a map $\widetilde{f}:P_{M}\rightarrow P_{N}$ such that
the following diagram become commutative:\\
\[
\begin{array}{ccccccccc}
0 & \longrightarrow & \Omega(M) & \longrightarrow & P_{M} & \longrightarrow & M & \longrightarrow & 0\\
\downarrow &  & \downarrow\widetilde{f}|_{\Omega(M)} &  & \downarrow\widetilde{f} &  & \downarrow f &  & \downarrow\\
0 & \longrightarrow & \Omega(N) & \longrightarrow & P_{N} & \longrightarrow & N & \longrightarrow & 0\end{array}\]
Now take $\Omega(f)$ to be the class of the induced map $\widetilde{f}|_{\Omega(f)}:\Omega(M)\rightarrow\Omega(N)$.
This is well defined by the following lemma:
\begin{lem}
1) In the previous notations, if $\tilde{f}_{1}$ and $\tilde{f}_{2}$
are two lifts of $f\circ\pi_{M}$ then $\tilde{f}_{1}|_{\Omega(M)}$
and $\tilde{f}_{2}|_{\Omega(M)}$ represent the same element in $\underline{Hom}_{R}(\Omega M,\Omega N)$\textup{.}\\
2) The map $Hom_{R}(M,N)\rightarrow\underline{Hom}_{R}(\Omega M,\Omega N)$
is a homomorphism.\\
3) If $f$ factors through a projective then also $\widetilde{f}|_{\Omega(f)}$
does, thus we get a homomorphism $\underline{Hom}_{R}(M,N)\rightarrow\underline{Hom}_{R}(\Omega M,\Omega N)$.\end{lem}
\begin{proof}
1) Assume we have two such lifts $\tilde{f}_{1}$ and $\tilde{f}_{2}$
then the following diagram is commutative (where $h=\tilde{f}_{1}|_{\Omega(M)}-\tilde{f}_{2}|_{\Omega(M)}$):\\
\[
\begin{array}{ccccccccc}
0 & \longrightarrow & \Omega(M) & \longrightarrow & P_{M} & \longrightarrow & M & \longrightarrow & 0\\
\downarrow &  & \downarrow h &  & \downarrow\tilde{f}_{1}-\tilde{f}_{2} &  & \downarrow0 &  & \downarrow\\
0 & \longrightarrow & \Omega(N) & \longrightarrow & P_{N} & \longrightarrow & N & \longrightarrow & 0\end{array}\]
It will be enough to show that $h$ factors through $P_{M}$ which
is projective. This follows from the fact that the image of the map
$\tilde{f}_{1}-\tilde{f}_{2}$ is contained in $\Omega(N)$ by the
commutativity of the diagram.\\
2) Choose the lifting of $a\cdot f+b\cdot g$ to be $a\cdot\widetilde{f}+b\cdot\widetilde{g}$.
\\
3) Assume $f$ factors through a projective module $P$. We have
the following diagram:\\
\[
\begin{array}{ccccccccc}
0 & \longrightarrow & \Omega(M) & \longrightarrow & P_{M} & \longrightarrow & M & \longrightarrow & 0\\
\downarrow &  & \downarrow &  & \downarrow &  & \downarrow &  & \downarrow\\
0 & \longrightarrow & 0 & \longrightarrow & P & \longrightarrow & P & \longrightarrow & 0\\
\downarrow &  & \downarrow &  & \downarrow s &  & \downarrow &  & \downarrow\\
0 & \longrightarrow & \Omega(N) & \longrightarrow & P_{N} & \longrightarrow & N & \longrightarrow & 0\end{array}\]
The map $s:P\rightarrow P_{N}$ can be defined using the fact that
$P$ is projective and the map $P_{N}\rightarrow N$ is surjective.
We get that the induced map $\Omega(M)\rightarrow\Omega(N)$ is the
zero map.
\end{proof}
The following is important for the definition of Tate cohomology: 
\begin{prop}
Let $G$ be a finite group and $R=\mathbb{Z}[G]$. If $M$ is a $\mathbb{Z}[G]$
module which is projective as an Abelian group then the map $\underline{Hom}_{R}(M,N)\rightarrow\underline{Hom}_{R}(\Omega M,\Omega N)$
is an isomorphism. \end{prop}
\begin{proof}
Before we start recall (\cite{key-3} VI 2) that a $\mathbb{Z}[G]$-module
$Q$ is called relatively injective if for every injection $A\hookrightarrow B$
of $\mathbb{Z}[G]$-modules which splits as an injection of Abelian
groups and every $\mathbb{Z}[G]$ homomorphism $A\rightarrow Q$ there
exists an extension to a $\mathbb{Z}[G]$ homomorphism $B\rightarrow Q$,
and that if $G$ is a finite group every projective module is relatively
injective.\\
We construct an inverse to this map. Given a map $f:\Omega M\rightarrow\Omega N$.
We have the following diagram:

\[
\begin{array}{ccccccccc}
0 & \longrightarrow & \Omega(M) & \longrightarrow & P_{M} & \longrightarrow & M & \longrightarrow & 0\\
\downarrow &  & \downarrow f\\
0 & \longrightarrow & \Omega(N) & \longrightarrow & P_{N} & \longrightarrow & N & \longrightarrow & 0\end{array}\]

Since $M$ is projective as an Abelian group the upper row splits
as Abelian groups. This means that $\Omega(M)\longrightarrow P_{M}$
is a split injection as Abelian groups. $P_{N}$ is projective and
hence relatively injective therefore we can extend the homomorphism
$\Omega(M)\longrightarrow P_{N}$ to a homomorphism $\tilde{f}:P_{M}\rightarrow P_{N}$
such that the diagram will commute. This induces a homomorphism $\overline{f}:M\rightarrow N$.
Of course $\overline{f}$ depends on the choice of $\tilde{f}$. Suppose
that $\tilde{f}_{1},\tilde{f_{2}}$ are two extensions then $\tilde{f_{1}}-\tilde{f_{2}}$
vanishes on $\Omega(M)$ hence the map $\overline{f}_{1}-\overline{f}_{2}:M\rightarrow N$
factors through $P_{N}$ which is projective. This gives a well defined
homomorphism $Hom_{R}(\Omega M,\Omega N)\rightarrow\underline{Hom}_{R}(M,N)$.
Assume $f:\Omega M\rightarrow\Omega N$ factors through a projective
$P$ then we can choose $\tilde{f}$ to factor through $P$ again
since it is relatively injective and get that $\overline{f}$ is the
zero map:

\[
\begin{array}{ccccccccc}
0 & \longrightarrow & \Omega(M) & \longrightarrow & P_{M} & \longrightarrow & M & \longrightarrow & 0\\
\downarrow &  & \downarrow &  & \downarrow &  & \downarrow &  & \downarrow\\
0 & \longrightarrow & P & \longrightarrow & P & \longrightarrow & 0 & \longrightarrow & 0\\
\downarrow &  & \downarrow &  & \downarrow s &  & \downarrow &  & \downarrow\\
0 & \longrightarrow & \Omega(N) & \longrightarrow & P_{N} & \longrightarrow & N & \longrightarrow & 0\end{array}\]

Hence we get a homomorphism $\underline{Hom}_{R}(\Omega M,\Omega N)\rightarrow\underline{Hom}_{R}(M,N)$
which is easily seen to be the inverse of the homomorphism $\underline{Hom}_{R}(M,N)\rightarrow\underline{Hom}_{R}(\Omega M,\Omega N)$.
\end{proof}
We have defined the endofunctor $\Omega$. We define $\Omega^{n}$
by induction: $\Omega^{0}=Id$ and $\Omega^{n}=\Omega\circ\Omega^{n-1}$.
\begin{prop}
Let $M$ be an $R$-module and let $...\rightarrow Q_{n-1}\rightarrow...\rightarrow Q_{0}\rightarrow M$
be any projective resolution of $M$, then $\Omega^{n}(M)$ can be
identified with $Ker(Q_{n-1}\rightarrow Q_{n-2})$, that is there
is a canonical map $Ker(Q_{n-1}\rightarrow Q_{n-2})\rightarrow\Omega^{n}(M)$
which is an isomorphism in the category $St-mod(R)$.\end{prop}
\begin{proof}
Given an $R$-module $M$ we construct a canonical projective resolution
of it using the projective covers we have chosen before. We do it
by induction where $P_{n}$ is defined to be the projective cover
of $Ker(P_{n-1}\rightarrow P_{n-2})$ with the induced map $P_{n}\rightarrow P_{n-1}$,
which clearly make this into a projective resolution. Notice that
by the definition of $\Omega$ we have $\Omega^{n}(M)=Ker(P_{n-1}\rightarrow P_{n-2})$,
and for a map $f:M\rightarrow N$ the map $\Omega^{n}(f)$ can be
be constructed by extending the map $f$ to a chain map between the
two resolutions. In order to prove the proposition it will suffice
to show that given two projective resolutions of $M$ $...\rightarrow Q_{n-1}\rightarrow...\rightarrow Q_{0}\rightarrow M$
and $...\rightarrow P_{n-1}\rightarrow...\rightarrow P_{0}\rightarrow M$
there is a canonical isomorphism $Ker(Q_{n-1}\rightarrow Q_{n-2})\rightarrow Ker(P_{n-1}\rightarrow P_{n-2})$.
This follows directly by induction from what we have already showed
in the case of a the projective cover of $M$. \end{proof}
\begin{rem*}
By similar reasons we can compute the induced maps $\Omega^{n}(f)$
for any map $f:M\rightarrow N$ by taking any two resolutions for
$M$ and for $N$ and extending $f$ into a chain map between the
two resolutions.\end{rem*}

\end{document}